\documentclass[leqno]{article}

\usepackage[frenchb,english]{babel}
\usepackage[T1]{fontenc}

\usepackage{amsmath}
\usepackage{amssymb}
\usepackage{amsfonts}
\usepackage{enumerate}
\usepackage{vmargin}
\usepackage[all]{xy}
\usepackage{mathrsfs}
\usepackage{mathtools}

\setmarginsrb{3cm}{3.25cm}{2.5cm}{2cm}{0cm}{0cm}{1.5cm}{3cm}

\newcommand{\Z}{\ensuremath{\mathbb{Z}}}
\newcommand{\R}{\ensuremath{\mathbb{R}}}
\newcommand{\C}{\ensuremath{\mathbb{C}}}
\newcommand{\T}{\ensuremath{\mathbb{T}}}

\renewcommand{\leq}{\ensuremath{\leqslant}}
\renewcommand{\geq}{\ensuremath{\geqslant}}
\newcommand{\n}{\noindent}
\newcommand{\qed}{\hfill \vrule height6pt  width6pt depth0pt}
\newcommand{\norm}[1]{ \| #1  \|}
\newcommand{\bnorm}[1]{ \big\| #1  \big\|}
\newcommand{\Bnorm}[1]{ \Big\| #1  \Big\|}
\newcommand{\bgnorm}[1]{ \bigg\| #1  \bigg\|}
\newcommand{\Bgnorm}[1]{ \Bigg\| #1  \Bigg\|}
\newcommand{\xra}{\xrightarrow}

\newcommand{\ot}{\otimes}
\newcommand{\ovl}{\overline}

\newcommand{\Rad}{{\rm Rad}}

\newtheorem{thm}{Theorem}[section]

\newtheorem{prop}[thm]{Proposition}

\newtheorem{lemma}[thm]{Lemma}

\newtheorem{remark}[thm]{Remark}

\newenvironment{proof}[1][]{\noindent {\it Proof #1} : }{\hbox{~}\qed
\smallskip
}

\numberwithin{equation}{section}

\begin{document}
\selectlanguage{english}
\title{\bfseries{Unconditionality, Fourier multipliers and Schur multipliers}}
\date{}
\author{\bfseries{C\'edric Arhancet}}

\maketitle

\begin{abstract}
Let $G$ be an infinite locally compact abelian group. If $X$ is
Banach space, we show that if every bounded Fourier multiplier $T$
on $L^2(G)$ has the property that $T\ot Id_X$ is bounded on
$L^2(G,X)$ then the Banach space $X$ is isomorphic to a Hilbert
space. Moreover, if $1<p<\infty$, $p\not=2$, we prove that there
exists a bounded Fourier multiplier on $L^p(G)$ which is not
completely bounded. Finally, we examine unconditionality from the
point of view of Schur multipliers. More precisely, we give several
necessary and sufficient conditions to determine if an operator
space is completely isomorphic to an operator Hilbert space.
\bigskip
\end{abstract}


\makeatletter
 \renewcommand{\@makefntext}[1]{#1}
 \makeatother
 \footnotetext{
 This work is partially supported by ANR 06-BLAN-0015 
}

\section{Introduction}
In \cite[Theorem 1]{DJ}, M. Defant and M. Junge proved the following
theorem (see also \cite[Theorem 1.5]{AB} and \cite[Theorem
8.4.11]{PiW}).
\begin{thm}
\label{Th incond T} Let $X$ be a Banach space. Suppose that there
exists a positive constant $C$ such that for any integer $n\in
\mathbb{N}$, any complex numbers $t_{-n},\ldots, t_n\in\C$ and any
$x_{-n},\ldots, x_n\in X$ we have
\begin{equation}\label{equa uncond T}
\Bgnorm{\sum_{k=-n}^{n}t_k e^{2\pi ik \cdot} \ot
x_k}_{L^2(\T,X)}\leq C \sup_{-n\leq k\leq n} |t_k|
\Bgnorm{\sum_{k=-n}^{n} e^{2\pi ik\cdot}\ot x_k}_{L^2(\T,X)}.
\end{equation}
Then the Banach space $X$ is isomorphic to a Hilbert space.
\end{thm}
This result says that if every bounded Fourier multiplier $T$ on
$L^2(\T)$ has the property that $T\ot Id_X$ is bounded on
$L^2(\T,X)$ then the Banach space $X$ is isomorphic to a Hilbert
space. The paper \cite[Theorem 1]{DJ} contains a generalization to
infinite compact abelian groups. Our first main result is an
extension of this theorem to infinite arbitrary locally compact
abelian groups.
\begin{thm}
\label{Th incond locallyintro} Let $G$ be an infinite locally
compact abelian group and $X$ be a Banach space. If every bounded
Fourier multiplier $T$ on $L^2(G)$ has the property that $T\ot Id_X$
is bounded on $L^2(G,X)$ then the Banach space $X$ is isomorphic to
a Hilbert space.
\end{thm}
Our proof is independent of the work \cite{DJ}.

Suppose $1\leq p \leq \infty$. We denote by $S^p=S^p(\ell^2)$ the
Schatten space. Let $\Omega$ be a measure space. Recall that a
linear map $T\colon L^p(\Omega)\xra{}L^p(\Omega)$ is completely
bounded if $T \ot Id_{S^p}$ extends to a bounded operator $T\ot
Id_{S^p}:L^p(\Omega,S^p)\xra{}L^p(\Omega,S^p)$, see \cite{Pis2}. In
this case, the completely bounded norm
$\norm{T}_{cb,L^p(\Omega)\xra{}L^p(\Omega)}$ is defined by
\begin{equation}\label{defnormecb}
\norm{T}_{cb,L^p(\Omega)\xra{}L^p(\Omega)}=\bnorm{T\ot Id_{S^p}
}_{L^p(\Omega,S^p)\xra{} L^p(\Omega,S^p)}.
\end{equation}
Let $G$ be a locally compact abelian group. If $p=1, 2$ or $\infty$,
it is easy to see that every bounded Fourier multiplier is
completely bounded on $L^p(G)$. If $1<p<\infty$, $p\not=2$, the
situation is different. Indeed, G. Pisier showed the following
theorem (see \cite[Proposition 8.1.3]{Pis2}, \cite[page 181]{Pis3}
and also \cite[Proposition 3.1]{Har}).
\begin{thm}
\label{Th non cb compact} Suppose $1<p<\infty$, $p\not=2$. Let $G$
be an infinite compact abelian group. There exists a bounded Fourier
multiplier on $L^p(G)$ which is not completely bounded.
\end{thm}
The author  \cite[Theorems 3.4 and 3.5]{Arh1} has given variants of
this result by proving the next theorem:
\begin{thm}
\label{Th noncb on Z and R} Suppose $1<p<\infty$, $p\not=2$. If
$G=\R$ or $G=\Z$, there exists a bounded Fourier multiplier on
$L^p(G)$ which is not completely bounded.
\end{thm}
In this paper, we give an extension of these both theorems to
arbitrary infinite locally compact abelian groups. Our second
principal result is the following.
\begin{thm}
\label{Th non cb locally} Suppose $1<p<\infty$, $p\not=2$. Let $G$
be an infinite locally compact abelian group. There exists a bounded
Fourier multiplier on $L^p(G)$ which is not completely bounded.
\end{thm}
The proof of this theorem and the one of Theorem \ref{Th non cb
compact} use a form of conditionality (i.e. non unconditionality).

If $1<p<\infty$ and if $E$ is an operator space, let $S^p(E)$ denote
the vector-valued noncommutative $L^p$-space defined in \cite{Pis2}.
The readers are referred to \cite{Pis2} and \cite{Pis3} for details
on operator spaces and completely bounded maps. For any index set
$I$, we denote by $OH(I)$ the associated operator Hilbert space
introduced by G. Pisier, see \cite{Pis3} and \cite{Pis4} for more
information. For any integers $i,j\geq 1$, let $e_{ij}$ be the
element of $S^p$ corresponding to the matrix with coefficients equal
to one at the $(i,j)$ entry and zero elsewhere. In the last section,
we show some results linked with unconditionality in the spirit of
Theorem \ref{Th incond T}. The following result is proved.
\begin{thm}
\label{Th incond Schurintro} Let $E$ be an operator space. The
following assertions are equivalent.
\begin{itemize}
  \item There exists a positive constant $C$ such that
\begin{equation*}
\Bgnorm{\sum_{i,j=1}^n t_{ij}e_{ij}\ot x_{ij}}_{S^2(E)}\leq C
\sup_{1\leq i,j\leq n} |t_{ij}| \Bgnorm{\sum_{i,j=1}^n e_{ij}\ot
x_{ij}}_{S^2(E)}
\end{equation*}
for any $n\in \mathbb{N}$, any complex numbers $\ t_{ij}\in \C$ and
any $\ x_{ij}\in E$.
  \item The operator space $E$ is completely isomorphic to an operator Hilbert space $OH(I)$ for some index set $I$.
\end{itemize}
\end{thm}

The paper is organized as follows. Section 2 gives preliminaries on
probability theory, Fourier multipliers and groups. We state some
results which are relevant to our paper. The next Section 3 contains
the proof of Theorem \ref{Th incond locallyintro}. In Section 4, we
give a proof of Theorem \ref{Th non cb locally}. Section 5 is
devoted to unconditionality from the point of view of Schur
multipliers. We present a proof of Theorem \ref{Th incond
Schurintro}.

Later in the paper, we will use $\lesssim$ to indicate an inequality
up to a constant which does not depend on the particular elements to
which it applies. Moreover $A(x)\approx B(x)$ will mean that we both
have $A(x)\lesssim B(x)$ and $B(x)\lesssim A(x)$.
\section{Preliminaries}
Let us recall some basic notations. If $A$ is a subset of a set $E$,
we let $1_A$ be the characteristic function of $A$. Let $\T=\big\{z
\in \C\ |\ |z|=1 \big\}$ and let $\Omega_0=\{-1,1\}^\infty$ be the
Cantor group equipped with their normalized  Haar measure. For any
integer $i\geq 1$, we define $\varepsilon_i$ by
$\varepsilon_i(\omega)=\omega_i$ if $\omega=(\omega_k)_{k\geq 1}\in
\Omega_0$. We can see the $\varepsilon_i$'s as independent
Rademacher variables on the probability space $\Omega_0$. Let $X$ be
a Banach space. Suppose $1<p<\infty$. We let $\Rad_p(X)\subset
L^p(\Omega_0,X)$ be the closure of ${\rm
Span}\bigl\{\varepsilon_i\otimes x\ |\ i\geq 1,\ x\in X\bigr\}$ in
the Bochner space $L^p(\Omega_0,X)$. Thus, for any finite family
$x_1,\ldots,x_n$ in $X$, we have
\begin{equation*}
\Bgnorm{\sum_{i=1}^{n} \varepsilon_i\otimes x_i}_{\Rad_p(X)} \,=\,
\Bigg(\int_{\Omega_0}\bgnorm{\sum_{i=1}^{n} \varepsilon_i(\omega)\,
x_i}_{X}^{p}\,d\omega\,\Bigg)^{\frac{1}{p}}.
\end{equation*}
We let $\Rad(X)=\Rad_2(X)$. By Kahane's inequalities (see e.g.
\cite[Theorem 11.1]{DJT}), the Banach spaces $\Rad(X)$ and
$\Rad_p(X)$ are canonically isomorphic.

We say that a set $F\subset B(X)$ is $R$-bounded provided that there
exists a constant $C\geq 0$ such that for any finite families
$T_1,\ldots, T_n$ in $F$ and $x_1,\ldots,x_n$ in $X$, we have
\begin{equation*}
\Bgnorm{\sum_{i=1}^{n} \varepsilon_i \ot T_i
(x_i)}_{\Rad(X)}\,\leq\, C\, \Bgnorm{\sum_{i=1}^{n}
\varepsilon_i\otimes x_i}_{\Rad(X)}.
\end{equation*}
$R$-boundedness was introduced in \cite{BG} and then developed in
the fundamental paper \cite{ClP}. We refer to the latter paper and
to \cite[Section 2]{KW} for a detailed presentation.

Recall that a Banach space $X$ has property $(\alpha)$ if there
exists a positive constant $C$ such that for any integer $n$, any
complex numbers $t_{ij}\in \C$ and any $x_{ij}$ in $X$ we have
\begin{equation*}
\Bgnorm{\sum_{i,j=1}^n t_{ij} \varepsilon_i\ot \varepsilon_{j}\ot
x_{ij}}_{\Rad(\Rad(X))} \leq C\ \sup_{1\leq i,j\leq n} |t_{ij}|\
\Bgnorm{\sum_{i,j=1}^n \varepsilon_i\ot \varepsilon_{j} \ot x_{ij}
}_{\Rad(\Rad(X))}.
\end{equation*}
If $1<p<\infty$, $p\not=2$, it is well-known that the space $S^p$
does not have property $(\alpha)$. If the Banach space $X$ has
property $(\alpha)$ and if $\Omega$ is a $\sigma$-finite measure
space then, for any $1<p<\infty$, the space $L^p(\Omega,X)$ also has
property $(\alpha)$. See \cite{Pis1}, \cite[page 148]{ClP} and
\cite[page 127]{KW} for more information on this property.

Let $Y$ be a Banach space and let $u\colon Y \to B(X)$ be a bounded
map. We say that $u$ is $R$-bounded if the set $\big\{ u(y) \ | \
\norm{y}_{Y} \leq 1 \big\}$ is $R$-bounded. We recall a fact which
is highly relevant for our paper. This result is \cite[Corollary
2.19]{DPR} (see also \cite[Corollary 4.5]{KLM}).
\begin{thm}
\label{bounded homo is R bounded} Let $K$ be a compact topological
space and $X$ be a Banach space with property $(\alpha)$. Any
bounded homomorphism $u\colon C(K)\to B(X)$ is $R$-bounded.
\end{thm}

Now, we record the following elementary lemma for later use. Its
easy proof is left to the reader.
\begin{lemma}
\label{Lemma formule RAdRAd} Suppose $1<p<\infty$. Let $E$ be an
operator space. We have an equality
\begin{equation*}
\label{formule RAdRAd} \Bgnorm{\sum_{i,j=1}^{n}e_{ij}\ot
x_{ij}}_{S^p(E)}=\Bgnorm{\sum_{i,j=1}^{n}\varepsilon_i\ot\varepsilon_j\ot
e_{ij} \ot x_{ij} }_{\Rad(\Rad(S^p(E))}, \qquad n\in \mathbb{N},
x_{ij}\in E.
\end{equation*}
\end{lemma}

Let $G$ be a locally compact abelian group with dual group
$\widehat{G}$. If $H$ is a subgroup of $G$, we denote by $H^\perp$
the annihilator of $H$. The group $(H^\perp)^\perp$ is equal to the
closure $\overline{H}$ of $H$ in $G$. If $H$ is a closed subgroup of
$G$ and if $\pi:G \to G/H$ denotes the canonical map, the mapping
$\chi\mapsto \chi\circ \pi$ is an isomorphism of $\widehat{G/H}$
onto $H^\perp$. Note that if $G$ is a locally compact abelian group
and if $H$ is a closed subgroup of $G$, we have an isomorphism
$\widehat{G}/ H^{\perp}=\widehat{H}$ given by
$\overline{\chi}\mapsto \chi|H $ (see \cite[Theorem 24.11]{HR}). See
\cite{FO} and \cite{HR} for background on abstract harmonic
analysis.

Let $G$ be a compact abelian group. A sequence $(\gamma_i)_{i\geq
1}$ of $\widehat{G}$ is a Sidon set if there exists a positive
constant $C$ such that
\begin{equation*}
\sum_{i=1}^n |\alpha_i| \leq C \Bgnorm{\sum_{i=1}^n
\alpha_i\gamma_i}_{L^\infty(G)},\qquad n\in \mathbb{N},
\alpha_1,\ldots,\alpha_n\in \C.
\end{equation*}
A typical example for $G=\T$ is a Hadamard set, see e.g. $\big\{ 2^i
\colon i\geq 1\big\}$. See \cite{HR} and \cite{LoR} for more
information on Sidon sets. Recall the following theorem
\cite[Theorem 2.1]{Pis5}.
\begin{thm}
\label{Th transfer Sidon} Let $G$ be a compact abelian group and
$(\gamma_i)_{i\geq 1}$ a Sidon set in $\widehat{G}$. Let $X$ be a
Banach space. Suppose $1<p<\infty$. Then we have the equivalence
\begin{equation*}
\Bgnorm{\sum_{i=1}^{n}\varepsilon_{i} \ot x_i}_{\Rad(X)} \approx
\Bgnorm{\sum_{i=1}^{n}\gamma_{i} \ot x_i}_{L^p(G,X)}, \qquad n\in
\mathbb{N}, x_1,\ldots,x_n\in X.
\end{equation*}
\end{thm}
Let $(\gamma_i)_{i\geq 1}$ be a Sidon set in $\widehat{G}$ where $G$
is a compact abelian group. Let $P$ be the orthogonal projection
from $L^2(G)$ onto the closed span of $\{ \gamma_i\ |\ i\geq 1 \}$
in the Hilbert space $L^2\big(G\big)$. Suppose $1<p<\infty$. It is
well-known that the restriction of $P$ to $ L^2(G)\cap L^p(G)$
extends to a bounded projection from $L^p(G)$ on the closure of
${\rm Span}\{\gamma_i\ |\ i\geq 1 \}$ in the space $L^p(G)$.

In the sequel, for any integer $q$, we consider the abelian group
$\oplus_{1}^\infty \mathbb{Z}/q\mathbb{Z}$ equipped with the
discrete topology. By \cite[Theorem 23.22 and page 367]{HR}, the
dual group of $\oplus_{1}^\infty \mathbb{Z}/q\mathbb{Z}$ is
isomorphic to the compact group
$\Pi_{1}^{\infty}\mathbb{Z}/q\mathbb{Z}$.

For any integer $i\geq 1$, we define the character
$\varepsilon_{i,q}$ of the group
$\Pi_{1}^{\infty}\mathbb{Z}/q\mathbb{Z}$ by
$\varepsilon_{i,q}\big(\ovl{k_1},\ldots,\ovl{k_j},\ldots\big)=e^{\frac{2\pi
\sqrt{-1} k_i}{q}}$ where $(k_j)_{j\geq 1}$ is a sequence of
integers of $\Z$. The compact group
$\Pi_{1}^{\infty}\mathbb{Z}/q\mathbb{Z}$ is an example of Vilenkin
group and the set of all characters of this group is called the
associated Vilenkin system. For more information, we refer the
reader to \cite[Appendix 0.7]{SWS} and the references contained
therein.

We will use the following lemma left to the reader.
\begin{lemma}
\label{Lemma epsq is a Sidon set} Let $q\geq 2$ be an integer. The
sequence $\big(\varepsilon_{i,q}\big)_{i\geq 1}$ of characters of
the group $\Pi_{1}^{\infty}\mathbb{Z}/q\mathbb{Z}$ is a Sidon set.
\end{lemma}
The sequence $\big(\varepsilon_{i,q}\big)_{i\geq 1}$ can be regarded
as a sequence of independent complex random variables on the
probability space $\Pi_{1}^{\infty}\mathbb{Z}/q\mathbb{Z}$. For any
integer $n$ and $q$, we introduce the compact finite group
$\Omega_q^n=\mathbb{Z}/q\mathbb{Z}\times\cdots
\times\mathbb{Z}/q\mathbb{Z}$. Note that $\Omega_{q}^n$ is a
subgroup of $\Pi_{1}^{\infty}\mathbb{Z}/q\mathbb{Z}$. The
restrictions $\varepsilon_{i,q}|\Omega_{q}^n$ to $\Omega_{q}^n$ of
the $\varepsilon_{i,q}$'s, where $1\leq i\leq n$, are characters of
the group $\Omega_q^n$ (see \cite[Theorem 23.21]{HR}) which can also
be regarded as a finite sequence of independent complex random
variables on the probability space $\Omega_q^n$.

We only require the use of averages of these random variables.
Moreover, if $X$ is a Banach space and $1<p<\infty$, these averages
are identical:
\begin{equation*}
\Bgnorm{\sum_{i=1}^n \varepsilon_{i,q}|\Omega_{q}^n \ot
x_{i}}_{L^p(\Omega_{q}^n,X)}= \Bgnorm{\sum_{i=1}^n \varepsilon_{i,q}
\ot x_{i}}_{L^p(\Pi_{1}^{\infty}\mathbb{Z}/q\mathbb{Z},X)},\qquad
n\in \mathbb{N}, x_{1},\ldots, x_{n}\in X.
\end{equation*}
Thus, if $n$ and $q$ are integers and $1\leq i\leq n$, we will use
also the notation $\varepsilon_{i,q}$ for the restriction
$\varepsilon_{i,q}|\Omega_{q}^n$. 

Suppose $1<p<\infty$. An operator $T\colon L^p(G) \xra{} L^p(G)$ is
a Fourier multiplier if there exists a function $\varphi \in
L^\infty\big(\widehat{G}\big)$ such that for any $f \in L^p(G)\cap
L^2(G)$ we have $\mathcal{F}\big(T(f)\big)= \varphi\mathcal{F}(f)$
where $\mathcal{F}$ denotes the Fourier transform. In this case, we
let $T=M_{\varphi}$. We denote by $M_p(G)$ the space of bounded
Fourier multipliers on $L^p(G)$. See \cite{Lar} and \cite{Der} for
more information. Let $X$ be a Banach space. The space $M_p(G,X)$ is
the space of bounded Fourier multipliers $M_\varphi$ such that
$M_\varphi\ot Id_X$ extends to a bounded operator $M_\varphi\ot
Id_{X}:L^p(G,X)\xra{}L^p(G,X)$. With these definitions and by
(\ref{defnormecb}), the space $M_p(G,S^p)$ coincides with the space
of completely bounded Fourier multipliers.

If $b\in L^1(G)$, we define the convolution operator $C_b$ by
\begin{equation*}
\begin{array}{cccc}
  C_b:  &  L^p(G)   &  \longrightarrow   & L^p(G)   \\
        &   f  &  \longmapsto       &  b*f.  \\
\end{array}
\end{equation*}
This operator is a completely bounded Fourier multiplier and we have
$C_b=M_{\mathcal{F}(b)}$. We will use the following approximation
result \cite[Theorem 5.6.1]{Lar} (see also \cite[Corollary 4 page
98]{Der}).
\begin{thm}
\label{Th Larsen} Suppose $1< p<\infty$. Let $G$ be a locally
compact abelian group. Let $M_{\varphi}\colon L^p(G) \xra{} L^p(G)$
be a bounded Fourier multiplier. Then there exists a net of
continuous functions $(b_i)_{i\in I }$ with compact support such
that
\begin{equation*}
\bnorm{C_{b_{i}}}_{L^p(G) \xra{}L^p(G)}\leq
\bnorm{M_{\varphi}}_{L^p(G) \xra{}L^p(G)} \ \ \ \text{and} \ \ \
C_{b_{i}} \xra[i]{so} M_{\varphi}
\end{equation*}
(convergence for the strong operator topology).
\end{thm}

We need the following vectorial extension of \cite[Theorem 2 page
113]{Der} (see also \cite[Theorem 3.3]{Sae}). We can prove this
result with a similar proof.
\begin{thm}
\label{Th transfer subgroups} Let $G$ be a locally compact abelian
group, $H$ be a closed subgroup of $G$ and $X$ be a Banach space. We
denote by $\pi: \widehat{G}\to \widehat{G}/ H^{\perp} $ the
canonical map. Then the linear map
\begin{equation*}
\begin{array}{cccc}
       &   M_p(H,X)   &  \longrightarrow   &  M_p(G,X)  \\
       &  M_\varphi   &  \longmapsto       &  M_{\varphi \circ \pi}  \\
\end{array}
\end{equation*}
is an isometry.
\end{thm}

The following proposition is well-known, see e.g. \cite[page
57]{FO}.
\begin{prop}{(Weil's formula)}
\label{Prop Weil formula}
Let $G$ a locally compact abelian group
and $H$ a closed subgroup of $G$. For any Haar measures $\mu_G$ and
$\mu_H$ on $G$ and $H$, respectively, there exists a Haar measure
$\mu_{G/H}$ on the group $G/H$ such that for every continuous
function $f\colon G\to\C$ with compact support we have
\begin{equation*}
\int_{G}
f(x)d\mu_G(x)=\int_{G/H}\int_{H}f(xh)d\mu_H(h)d\mu_{G/H}(xH).
\end{equation*}
\end{prop}
With this result, we can prove the next proposition. 
\begin{prop}
\label{Prop transfer quotient} Suppose $1<p<\infty$. Let $G$ be a
locally compact abelian group, $H$ be a compact subgroup of $G$ and
$X$ be a Banach space. If $\varphi\colon H^{\perp} \to \C$ is a
complex function, we denote by
$\widetilde{\varphi}\colon\widehat{G}\to \C$ the extension of
$\varphi$ on $\widehat{G}$ which is zero off $H^{\perp}$. Then the
linear map
\begin{equation*}
\begin{array}{cccc}
    &   M_p(G/H,X)  &  \longrightarrow   &   M_p(G,X)   \\
    &   M_{\varphi}  &  \longmapsto       &   M_{\widetilde{\varphi}} \\
\end{array}
\end{equation*}
is an isometry.
\end{prop}

\begin{proof}
We denote $\pi \colon G\to G/H$ the canonical map. We use the Haar
measures given by Proposition \ref{Prop Weil formula}. We can
suppose that $\mu_H(H)=1$. Using the Weil's formula, it is not
difficult to prove that the linear map
\begin{equation*}
\begin{array}{cccc}
   \Phi_p: &   L^p(G/H)  &  \longrightarrow   &   L^p(G)   \\
    &   f  &  \longmapsto       &   f\circ \pi \\
\end{array}
\end{equation*}
and its tensorisation $\Phi_p\ot Id_{X}\colon L^p(G/H,X) \to
L^p(G,X)$ are isometries. Note that the adjoint map $\Phi_{p^*}^*$
and the orthogonal projection of $L^2(G)$ onto
$\Phi_2\big(L^2(G/H)\big)$ coincide on $L^2(G)\cap L^{p}(G)$.
Moreover, it is easy to see that the linear map $\Phi_{p^*}^*\ot
Id_{X}$ is well-defined and contractive. The end of the proof is
straightforward and left to the reader.
\end{proof}

Recall the following structure theorem for locally compact abelian
groups, see e.g. \cite[Theorem 24.30]{HR}.
\begin{thm}
\label{Th structure} Any locally compact abelian group is isomorphic
to a product $\R^n\times G_0$ where $n\geq 0$ is an integer and
$G_0$ is a locally compact abelian group containing a compact
subgroup $K$ such that $G_0/K$ is discrete.
\end{thm}
Let $(G_i)_{i\in I}$ be a family of groups and let $\Pi_{i\in I}
G_i$ be the cartesian product of the groups $G_i$. Recall that the
direct sum $\oplus_{i\in I} G_i$ of the group $G_i$ is the set of
all $(x_i)_{i\in I}\in \Pi_{i\in I} G_i$ such that $x_i=e_i$ for all
but a finite set of indices where $e_i$ is the neutral element of
$G_i$. The group $\oplus_{i\in I} G_i$ is a subgroup of $\Pi_{i\in
I} G_i$. Recall that a group of bounded order is a group such that
every element has finite order and the order of each element is less
than some fixed positive integer. Note the next result \cite[page
449]{HR}.
\begin{thm}
\label{Th struct group bounded} Every abelian group $G$ (without
topology) of bounded order is isomorphic to a direct sum
$\oplus_{i\in I} \mathbb{Z}/q_i^{r_i}\mathbb{Z}$ of cyclic groups,
where only finitely many distinct
primes $q_i$ and positive integers $r_i$ occur. 
\end{thm}
This theorem implies that an infinite abelian group $G$ of bounded
order contains a direct sum $\oplus_{1}^{\infty}
\mathbb{Z}/q\mathbb{Z}$ where $q$ is a fixed prime.

\section{Unconditionality and Fourier multipliers}
Suppose $1 < p < \infty$. Let $G$ be a locally compact group and $X$
a Banach space. If $t\in G$, we denote by $\tau_t$ the translation
operator on $L^p(G)$ defined by $\tau_t(f)(s)=f(t^{-1}s)$ where $f
\in L^p(G)$ and $s\in G$. We start with the next result.
\begin{lemma}
\label{Lemma tau imply Hilbert} Let $G$ be an infinite locally
compact group and $X$ a Banach space.  If the set $\big\{\tau_t \ot
Id_{X}\ | \ t\in G \big\}$ is $R$-bounded in $B\big(L^2(G,X)\big)$
then the Banach space $X$ is isomorphic to a Hilbert space.
\end{lemma}

\begin{proof}
Let $n\geq 1$ be an integer and $t_1,\ldots,t_n$ be distinct
elements of $G$. There exists a compact neighborhood $V$ of the
neutral element $e_G$ of $G$ such that the sets $t_1V,\ldots,t_nV$
are disjoint. We have $\mu_G(V)>0$. For any integer $1\leq i\leq n$,
we let $V_i=t_iV$. First note that, for any $x_1,\ldots,x_n\in X$,
we have
\begin{align}
\label{serie equations}
 \Bigg(\sum_{i=1}^{n} \bnorm{1_{V_i}}_{L^2(G)}^2 \norm{x_i}_{X}^2\Bigg)^{\frac{1}{2}}
 &= \Bigg(\int_{\Omega_0} \sum_{i=1}^{n}\bnorm{\varepsilon_i(\omega)1_{V_i}\ot x_i}_{L^2(G,X)}^2d\omega\Bigg)^{\frac{1}{2}}  \nonumber\\
 &= \Bigg(\int_{\Omega_0} \Bgnorm{\sum_{i=1}^{n}\varepsilon_i(\omega)1_{V_i}\ot x_i}_{L^2(G,X)}^2d\omega\Bigg)^{\frac{1}{2}}\hspace{0.4cm}\text{since the $V_i$'s are disjoint}\nonumber\\
 &=\Bgnorm{\sum_{i=1}^{n}\varepsilon_i \ot 1_{V_i} \ot x_{i}}_{\Rad(L^2(G,X))}.
\end{align}
We deduce that
\begin{align*}
 \Bigg(\sum_{i=1}^{n} \bnorm{1_{V_i}}_{L^2(G)}^2 \norm{x_i}_{X}^2\Bigg)^{\frac{1}{2}}
    &= \Bgnorm{\sum_{i=1}^{n}\varepsilon_i \ot (\tau_{t_i}\ot Id_{X})(1_{V} \ot x_{i})}_{\Rad(L^2(G,X))}\\
    &\lesssim \Bgnorm{\sum_{i=1}^{n}\varepsilon_i \ot 1_{V} \ot x_{i}}_{\Rad(L^2(G,X))}\\
    &= \norm{1_{V}}_{L^2(G)}\Bgnorm{\sum_{i=1}^{n}\varepsilon_i\ot x_{i}}_{\Rad(X)}.
\end{align*}
For any integer $1\leq i\leq n$, we have
$\norm{1_{V}}_{L^2(G)}=\bnorm{1_{V_i}}_{L^2(G)}$. We infer that
\begin{align*}
 \Bigg(\sum_{i=1}^{n} \norm{x_i}_{X}^2\Bigg)^{\frac{1}{2}}
    &\lesssim \Bgnorm{\sum_{i=1}^{n}\varepsilon_i\ot x_{i}}_{\Rad(X)}.
\end{align*}
We deduce that $X$ has cotype 2. Now, for any $x_1,\ldots,x_n\in X$,
we have
\begin{align*}
\bnorm{1_{V}}_{L^2(G)} \Bgnorm{\sum_{i=1}^{n}\varepsilon_i\ot
x_{i}}_{\Rad(X)}
    &= \Bgnorm{\sum_{i=1}^{n}\varepsilon_i \ot 1_{V} \ot x_{i}}_{\Rad(L^2(G,X))}\\
    &= \Bgnorm{\sum_{i=1}^{n}\varepsilon_i \ot (\tau_{t_i^{-1}}\ot Id_{X})(\tau_{t_i}\ot Id_{X})(1_{V} \ot x_{i})}_{\Rad(L^2(G,X))}\\
    &\lesssim \Bgnorm{\sum_{i=1}^{n}\varepsilon_i \ot 1_{V_i} \ot x_{i}}_{\Rad(L^2(G,X))}\\
    &= \Bigg(\sum_{i=1}^{n} \bnorm{1_{V_i}}_{L^2(G)}^2 \norm{x_i}_{X}^2\Bigg)^{\frac{1}{2}}\hspace{0.4cm}\text{by (\ref{serie equations})}.
\end{align*}
Using, one more time, the equality
$\norm{1_{V}}_{L^2(G)}=\bnorm{1_{V_i}}_{L^2(G)}$ for any integer
$1\leq i\leq n$, we deduce that
\begin{equation*}
\Bgnorm{\sum_{i=1}^{n}\varepsilon_i\ot x_{i}}_{\Rad(X)}\lesssim
\Bigg(\sum_{i=1}^{n} \norm{x_i}_{X}^2\Bigg)^{\frac{1}{2}}.
\end{equation*}
We deduce that $X$ has type 2. Hence, by  Kwapie\'{n}'s theorem
\cite[Proposition 3.1]{Kwa} (or \cite[Corollary 12.20]{DJT}), the
Banach space $X$ is isomorphic to a Hilbert space.
\end{proof}

Let $G$ be a locally compact abelian group and $X$ be a Banach
space. If $X$ is isomorphic to a Hilbert space, it is clear that we
have a canonical isomorphism $M_2(G,X)= M_2(G)$. We will show the
reverse implication for \emph{infinite} locally compact abelian
groups.

We begin with the case of $\T$. We give a proof which does not use
\cite{Def}.  We will use the elementary lemma left the reader.
\begin{lemma}
\label{Lemma passage limite} Let $g:\T\times \T\to \C$ be a
continuous complex function. We have
$$
\int_{\T}g\big(z,z^k\big)dz\xra[k \to +\infty]{}\int_{\T \times \T
}g(z,z')dzdz'.
$$
\end{lemma}
Now, we can prove the following proposition.
\begin{prop}
\label{Prop incond T} Let $X$ be a Banach space. We have a canonical
isomorphism $M_2(G,\T)= M_2(\T)$ if and only if the space $X$ is
isomorphic to a Hilbert space.
\end{prop}
\begin{proof}
Suppose that $M_2(\T,X)=M_2(\T)$. For any integer $i\geq 1$, we let
$n_i=2^{2i}$ and $m_i=2^{2i+1}$. The sequences $(n_i)_{i\geq 1}$ and
$(m_j)_{j\geq 1}$ are Sidon sets for the group $\T$. We will use the
fact that there exists arbitrary large integers $k\geq 1$ such the
map $(i,j) \to n_i+km_{j}$ is one-to-one. Note that, by Theorem
\ref{Th transfer Sidon}, we have an equivalence
\begin{equation}
\label{equivalence epseps hadhad} \Bgnorm{\sum_{i,j=1}^n
\varepsilon_i \ot \varepsilon_{j}\ot x_{ij}}_{\Rad(\Rad(X))}
\approx\ \Bgnorm{\sum_{i,j=1}^n e^{2\pi\sqrt{-1}n_i\cdot} \ot
e^{2\pi\sqrt{-1}m_j\cdot} \ot x_{ij}}_{L^2(\T \times \T,X)},\qquad
n\in\mathbb{N}, x_{ij}\in X.
\end{equation}

Now, suppose that the Banach space $X$ does not have property
$(\alpha)$. Let $C$ be a positive constant. Then there exists an
integer $n\geq 1$, complex numbers $t_{ij}\in \C$ with $|t_{ij}|=1$
and $x_{ij}\in X$ such that $\Bnorm{\sum_{i,j=1}^n \varepsilon_i\ot
\varepsilon_{j}\ot x_{ij}}_{\Rad(\Rad(X))} \leq 1 $ with arbitrary
large $\Bnorm{\sum_{i,j=1}^n t_{ij} \varepsilon_i\ot
\varepsilon_{j}\ot x_{ij}}_{\Rad(\Rad(X))}$. Using the equivalence
(\ref{equivalence epseps hadhad}), we deduce that there exists an
integer $n\geq 1$, complex numbers $t_{ij}\in \C$ with $|t_{ij}|=1$
and $x_{ij}\in X$ such that
$$
\Bgnorm{\sum_{i,j=1}^n e^{2\pi\sqrt{-1}n_i\cdot} \ot
e^{2\pi\sqrt{-1}m_j\cdot} \ot x_{ij}}_{L^2(\T \times \T,X)} \leq
\frac{1}{2}
$$
and
$$
\Bgnorm{\sum_{i,j=1}^n e^{2\pi\sqrt{-1}n_i\cdot} \ot
e^{2\pi\sqrt{-1}m_j\cdot} \ot x_{ij}}_{L^2(\T \times \T,X)} \geq 2C.
$$
Moreover, by Lemma \ref{Lemma passage limite}, we have
\begin{equation*}
\Bgnorm{\sum_{i,j=1}^n e^{2\pi\sqrt{-1}(n_i+km_j)\cdot} \ot
x_{ij}}_{L^2(\T,X)}\xra[k \to +\infty]{}\Bgnorm{\sum_{i,j=1}^n
e^{2\pi\sqrt{-1}n_i\cdot} \ot e^{2\pi\sqrt{-1}m_j\cdot} \ot
x_{ij}}_{L^2(\T \times \T,X)}.
\end{equation*}
For some $k$ large enough, we deduce the following inequalities
\begin{equation*}
\Bgnorm{\sum_{i,j=1}^n e^{2\pi\sqrt{-1}(n_i+km_j)\cdot}\ot
x_{ij}}_{L^2(\T,X)}\leq 1
\end{equation*}
and
\begin{equation*}
\Bgnorm{\sum_{i,j=1}^n t_{ij}e^{2\pi\sqrt{-1}(n_i+km_j)\cdot}\ot
x_{ij}}_{L^2(\T,X)}> C.
\end{equation*}
We infer that the inequality (\ref{equa uncond T}) is not satisfied.
Contradiction. Thus, the Banach space $X$ has property $(\alpha)$.

We deduce that the space $L^2(\T,X)$ also has property $(\alpha)$.
Now, note that $L^\infty(\T)$ is a commutative unital $C^*$-algebra.
By Gelfand's Theorem (see e.g. \cite[Theorem 1.20]{FO}), the Banach
algebra $L^\infty(\T)$ is isometrically isomorphic to a Banach
algebra $C(K)$ where $K$ is a compact topological space. Moreover,
we have a bounded homomorphism
\begin{equation*}
\begin{array}{cccc}
    &  L^\infty(\T)   &  \longrightarrow   & B\big(L^2(\T,X)\big)  \\
    &                            \varphi                                      &  \longmapsto       &   M_\varphi. \\
\end{array}
\end{equation*}
By Theorem \ref{bounded homo is R bounded}, we infer that this
linear map is $R$-bounded. For any $t\in G$, note that the map
$\tau_t$ is an isometric Fourier multiplier.  Hence the set
$\big\{\tau_t\ot Id_{X}\ \colon \ t\in \T \big\}$ is $R$-bounded. By
Lemma \ref{Lemma tau imply Hilbert}, we conclude that the Banach
space $X$ is isomorphic to a Hilbert space.
\end{proof}

Now, we extend Proposition \ref{Prop incond T} to the groups $\R$
and $\Z$. We use a method similar to the one of \cite[Theorems 3.4
and 3.5]{Arh1}. Since we need variants of this method later (and
also for the convenience of the reader), we include some details.
For that purpose, we need the following vectorial extension of
\cite[Proposition 3.3]{DeL}. One can prove this theorem as
\cite[Theorem 3.4]{CoW}.
\begin{thm}\label{deleuw}
Let $X$ be a Banach space. Suppose $1 < p < \infty$. Let $\psi$ be a
continuous function on $\R$ which defines a bounded Fourier
multiplier $M_\psi$ on $L^{p}(\R,X)$. Then the restriction $\psi|\Z$
of the function $\psi$ to $\Z$ defines a bounded
Fourier multiplier $M_{\psi|\Z}$ on $L^{p}(\T,X)$. 
\end{thm}
Moreover, we need the next result of Jodeit \cite[Theorem 3.5]{Jod}.
We introduce the function $\Lambda\colon \mathbb{R} \xra{}
\mathbb{R}$ defined by
\begin{equation*}
\Lambda(x)=\bigg\{\begin{array}{cl}
        1-|x|  &  {\rm if}\quad  x \in [-1,1]\\
         0     &  {\rm if}\quad  |x|>1.
       \end{array}
\end{equation*}
\begin{thm}
\label{Jodeit} Suppose $1<p<\infty$. Let $\varphi$ be a complex
function defined on $\Z$ such that $M_{\varphi}$ is a bounded
Fourier multiplier on $L^p(\T)$. Then the complex function
$\psi\colon \mathbb{R} \xra{} \mathbb{C}$ defined by
\begin{equation}\label{def psi}
\psi(x)=\sum_{k\in\mathbb{Z}} \varphi(k)\Lambda(x-k),\qquad x\in\R,
\end{equation}
defines a bounded Fourier multiplier $M_{\psi}$ on
$L^p(\mathbb{R})$.
\end{thm}
Now, we can prove the next Proposition.
\begin{prop}
\label{Th incond R and Z} Let $X$ be a Banach space. Suppose that
$G=\R$ or $G=\Z$. We have a canonical isomorphism $M_2(G,X)= M_2(G)$
if and only if the space $X$ is isomorphic to a Hilbert space.
\end{prop}

\begin{proof}
Suppose that $X$ is not isomorphic to a Hilbert space. By
Proposition \ref{Prop incond T}, there exists a bounded Fourier
multiplier $M_{\varphi}\colon L^2(\mathbb{T}) \xra{}
L^2(\mathbb{T})$ such that $M_{\varphi}\ot Id_{X}$ is not bounded on
$L^2(\mathbb{T},X)$. Now, consider the function $\psi$ given by
(\ref{def psi}). By Theorem \ref{Jodeit}, this function defines a
bounded Fourier multiplier $M_{\psi}\colon L^2(\mathbb{R}) \xra{}
L^2(\mathbb{R})$. Now, suppose that the map $M_{\psi}\ot
Id_{X}\colon L^2(\mathbb{R},X) \xra{} L^2(\mathbb{R},X)$ is bounded.
Since the function $\psi\colon \mathbb{R}\xra{}\mathbb{C}$ is
continuous, by Theorem \ref{deleuw}, we deduce that the restriction
$\psi|\Z$ defines a bounded Fourier multiplier $M_{\psi|\Z}$ on
$L^2(\mathbb{T},X)$. Moreover, we observe that, for any $k \in
\mathbb{Z}$, we have $ \psi(k)=\varphi(k)$. Then we deduce that the
Fourier multiplier $M_{\varphi}$ is bounded on $L^2(\mathbb{T},X)$.
We obtain a contradiction. Consequently, the Fourier multiplier
$M_{\psi}$ is bounded on $L^2(\mathbb{R})$ and $M_{\psi}\ot Id_{X}$
is not bounded on $L^2(\mathbb{R},X)$. Hence, the case $G=\R$ is
completed.

We can suppose that the above multiplier $M_{\psi}$ satisfies
$\bnorm{M_{\psi}}_{L^2(\mathbb{R}) \xra{}L^2(\mathbb{R})}=1$. By
Theorem \ref{Th Larsen}, there exists a net of continuous functions
$(b_i)_{i \in I}$ with compact support such that
$$
\bnorm{C_{b_{i}}}_{L^2(\mathbb{R}) \xra{} L^2(\mathbb{R})}\leq 1
 \ \ \ \text{and} \ \ \ C_{b_{i}} \xra[i]{so}
M_{\psi}.
$$
Let $C>1$. Then, it is not difficult to deduce that there exists a
continuous function $b\colon \mathbb{R} \xra{}\mathbb{C}$ with
compact support such that $\norm{C_b}_{L^2(\mathbb{R}) \xra{}
L^2(\mathbb{R})} \leq 1$ and $\bnorm{C_b \ot
Id_{X}}_{L^2(\mathbb{R},X) \xra{} L^2(\mathbb{R},X)} \geq 2C$.
Now, we define the sequence $\big(a_n\big)_{n\geq 1}$ of complex
sequences indexed by $\mathbb{Z}$ by, if $n\geq 1$ and $k \in
\mathbb{Z}$
\begin{equation}\label{ank}
    a_{n,k}=\int_{0}^{1}\int_{0}^{1}\frac{1}{n}b\bigg(\frac{t-s+k}{n}\bigg)dsdt.
\end{equation}
For any integer $n\geq 1$, we introduce the conditional expectation
$\mathbb{E}_n\colon L^2(\mathbb{R}) \xra{} L^2(\mathbb{R})$ with
respect to the $\sigma$-algebra generated by the
$\Big[\frac{k}{n},\frac{k+1}{n}\Big[$, $k \in \mathbb{Z}$. For any
integer $n\geq 1$ and any $f \in L^2(\mathbb{R})$, we have
\begin{equation*}
\mathbb{E}_nf=n \sum_{k \in
\mathbb{Z}}^{}\Bigg(\int_{\frac{k}{n}}^{\frac{k+1}{n}} f(t)dt\Bigg)
1_{\big[\frac{k}{n},\frac{k+1}{n}\big[}
\end{equation*}
(see \cite[page 227]{AbA}). Now, we define the linear map $J_n\colon
\ell^2_\mathbb{Z} \xra{} \mathbb{E}_n\big(L^2(\mathbb{R})\big)$ by,
if $u \in \ell^2_\mathbb{Z}$
\begin{equation*}
J_n(u)=n^{\frac{1}{2}}\sum_{k\in \mathbb{Z}} u_k
1_{\big[\frac{k}{n},\frac{k+1}{n}\big[}.
\end{equation*}
It is easy to check that the map $J_n$ is an isometry of
$\ell^2_\mathbb{Z}$ onto the range
$\mathbb{E}_n\big(L^2(\mathbb{R})\big)$ of $\mathbb{E}_n$. For any
$u\in \ell^2_\mathbb{Z}$, mimicking the computation presented in the
proof of \cite[Theorem 3.5]{Arh1}, we obtain that
\begin{equation*}
\mathbb{E}_nC_bJ_n(u)=J_nC_{a_n}(u).
\end{equation*}
Then, it is easy to prove that there exists an integer $n\geq 1$
such that $ \bnorm{C_{a_n}}_{\ell^2_\mathbb{Z}
\xra{}\ell^2_\mathbb{Z}} \leq 1 $ and $\bnorm{C_{a_n}\ot
Id_{X}}_{\ell^2_\mathbb{Z}(X) \xra{}\ell^2_\mathbb{Z}(X)} \geq C$.
Finally, we conclude the case $G=\Z$ with the closed graph theorem.
\end{proof}

Now, we pass to discrete groups. We first prove the following result
with a method similar to that of Proposition \ref{Prop incond T}.
\begin{prop}
\label{Prop incond oplus} Let $X$ be a Banach space. Let $q\geq 2$
be an integer. We have a canonical isomorphism
$M_2(\oplus_{1}^\infty \mathbb{Z}/q\mathbb{Z},X) =
M_2(\oplus_{1}^\infty \mathbb{Z}/q\mathbb{Z})$ if and only if the
Banach space $X$ is isomorphic to a Hilbert space.
\end{prop}

\begin{proof}
Assume that $M_2(\oplus_{1}^\infty \mathbb{Z}/q\mathbb{Z},X) =
M_2(\oplus_{1}^\infty \mathbb{Z}/q\mathbb{Z})$. Then there exists a
positive constant $C$ such that for any $\varphi\in
L^\infty(\Pi_{1}^{\infty}\mathbb{Z}/q\mathbb{Z})$
\begin{equation*}
\bnorm{M_{\varphi}}_{L^2(\oplus_{1}^\infty
\mathbb{Z}/q\mathbb{Z},X)\to L^2(\oplus_{1}^\infty
\mathbb{Z}/q\mathbb{Z},X)} \leq C \norm{\varphi}_{
L^\infty(\Pi_{1}^{\infty}\mathbb{Z}/q\mathbb{Z})}.
\end{equation*}
Moreover, if $n$ is an integer, note that $\Omega_{q}^n \times
\Omega_{q}^n$ is a closed subgroup of $\oplus_{1}^\infty
\mathbb{Z}/q\mathbb{Z}$. For any integer $n\geq 1$, any complex
numbers $t_{ij}\in \C$ and any $x_{ij}\in X$, we deduce that
\begin{equation*}
\Bgnorm{\sum_{i,j=1}^n t_{ij}\varepsilon_{i,q} \ot \varepsilon_{j,q}
\ot x_{ij}}_{L^2(\Omega_{q}^n \times \Omega_{q}^n,X)}\leq C
\sup_{1\leq i,j\leq n}|t_{ij}|\Bgnorm{\sum_{i,j=1}^n
\varepsilon_{i,q} \ot \varepsilon_{j,q} \ot
x_{ij}}_{L^2(\Omega_{q}^n \times \Omega_{q}^n,X)}.
\end{equation*}
Now, by Theorem \ref{Th transfer Sidon} and Lemma \ref{Lemma epsq is
a Sidon set}, we have the equivalence
\begin{equation*}
\Bgnorm{\sum_{i,j=1}^n \varepsilon_i \ot \varepsilon_{j}\ot
x_{ij}}_{\Rad(\Rad(X))} \approx\ \Bgnorm{\sum_{i,j=1}^n
\varepsilon_{i,q} \ot \varepsilon_{j,q} \ot
x_{ij}}_{L^2(\Omega_{q}^n \times \Omega_{q}^n,X)},\qquad
n\in\mathbb{N}, x_{ij}\in X.
\end{equation*}
We deduce that the Banach space $X$ has property $(\alpha)$. The end
of the proof is similar to the end of the proof of Proposition
\ref{Prop incond T}.
\end{proof}

\begin{prop}
\label{Th incond discrete} Let $G$ be an infinite discrete abelian
group and $X$ a Banach space. We have a canonical isomorphism
$M_2(G,X)=M_2(G)$ if and only if the space $X$ is isomorphic to a
Hilbert space.
\end{prop}

\begin{proof} Case 1: $G$ is not a torsion group. Then $G$ contains a copy of
$\mathbb{Z}$, the additive group of the integers. Suppose that
$M_2(G,X)=M_2(G)$. By Theorem \ref{Th transfer subgroups}, we have
$M_2(\Z,X)=M_2(\Z)$. By Proposition \ref{Th incond R and Z}, we
deduce that $X$ is isomorphic to a Hilbert space.

Case 2: $G$ is a torsion group, but contains elements of arbitrarily
large order. We may therefore assume that there is a sequence
$G_1,G_2,\ldots$ of cyclic subgroups of $G$ of orders
$n_1,n_2,\ldots$ with $n_j\xra[j \to +\infty]{} +\infty$.

We will construct contractive Fourier multipliers $C_{a_n}$ on the
cyclic group $\Z/n\Z$ with large $\bnorm{C_{a_n}\ot
Id_{X}}_{\ell^2_n(X) \xra{} \ell^2_n(X)}$. We use a similar method
to the one of proof of Proposition \ref{Th incond R and Z}. By
Proposition \ref{Prop incond T}, there exists a bounded Fourier
multiplier $M_{\psi}\colon L^2(\mathbb{T})\xra{} L^2(\mathbb{T})$
such that $M_{\psi}\ot Id_{X}$ is not bounded on
$L^2(\mathbb{T},X)$. By Theorem \ref{Th Larsen}, there exists a net
of continuous functions $(b_i)_{i \in I}$ such that
\begin{equation*}
\bnorm{C_{b_{i}}}_{L^2(\T) \xra{} L^2(\T)}\leq
\bnorm{M_{\psi}}_{L^2(\T) \xra{}L^2(\T)} \ \ \ \text{and} \ \ \
C_{b_{i}} \xra[i]{so} M_{\psi}.
\end{equation*}
Let $C>1$. It is not difficult to deduce that there exists a
continuous function $b\colon \mathbb{T} \xra{}\mathbb{C}$ such that
\begin{equation*}
\norm{C_b}_{L^2(\mathbb{T}) \xra{} L^2(\mathbb{T})} \leq 1 \ \ \ \ \
\text{and} \ \ \ \ \ \norm{C_b \ot Id_{X}}_{L^2(\mathbb{T},X) \xra{}
L^2(\mathbb{T},X)} \geq 2C.
\end{equation*}
Now, we use the identification $L^2(\mathbb{T})=L^2\big([0,1]\big)$.
We consider the function $b$ as a 1-periodic function $b\colon \R\to
\C$. Then, we define the sequence $\big(a_n\big)_{n\geq 1}$ of
complex sequences indexed by $\{0,\ldots,n\}$ by (\ref{ank}), if
$n\geq 1$ and $k \in \{0,\ldots,n\}$. If $n\geq 1$, note that the
map $C_{a_n}$ is a convolution operator on $\ell^2_n$. For any
integer $n\geq 1$, we introduce the conditional expectation
$\mathbb{E}_n\colon L^2\big([0,1]\big) \xra{} L^2\big([0,1]\big)$
with respect to the $\sigma$-algebra generated by the
$\Big[\frac{k}{n},\frac{k+1}{n}\Big[$, $k \in \{0,\ldots,n\}$. For
any integer $n\geq 1$ and any $f \in L^2\big([0,1]\big)$, we have
\begin{equation}\label{def esperance}
\mathbb{E}_nf=n
\sum_{k=0}^{n-1}\Bigg(\int_{\frac{k}{n}}^{\frac{k+1}{n}}
f(t)dt\Bigg) 1_{[\frac{k}{n},\frac{k+1}{n}[}.
\end{equation}
Now, we define the linear map $J_n\colon \ell^2_n\xra{}
\mathbb{E}_n\big( L^2([0,1])\big)$ by, if $u \in \ell^2_n$
\begin{equation*}
J_n(u)=n^{\frac{1}{p}}\sum_{k=0}^{n-1} u_k
1_{[\frac{k}{n},\frac{k+1}{n}[}.
\end{equation*}
It is easy to check that the map $J_n$ is an isometry of $\ell^2_n$
onto the range $\mathbb{E}_n\big(L^2([0,1])\big)$ of $\mathbb{E}_n$.
For any $u\in \ell_p^n$, by a computation similar to the one of the
proof of \cite[Theorem 3.5]{Arh1}, we show that
\begin{equation*}
\mathbb{E}_nC_bJ_n(u)= J_nC_{a_n}(u).
\end{equation*}
Thus, it is not difficult to deduce that there exists an integer
$N\geq 1$ such that for any integer $n\geq N$ we have
\begin{equation*}
\bnorm{C_{a_n}}_{\ell^2_n \xra{}\ell^2_n} \leq 1 \ \ \ \ \
\text{and} \ \ \ \ \ \bnorm{C_{a_n}\ot Id_{X}}_{\ell^2_n(X)
\xra{}\ell^2_n(X)} \geq C.
\end{equation*}

Now, recall that $n_j\xra[j \to +\infty]{} +\infty$. Hence, we
deduce that there exists an integer $j\geq 1$ and a convolution
operator $C_{a}\colon L^2(G_{n_j})\to L^2(G_{n_j})$ such that
$$
\bnorm{C_{a}}_{L^2(G_{n_j}) \xra{} L^2(G_{n_j})} \leq 1 \ \ \ \ \
\text{and} \ \ \ \ \ \bnorm{C_{a} \ot Id_{X}}_{L^2(G_{n_j},X)
\xra{}L^2(G_{n_j},X)} \geq C.
$$
We conclude with Theorem \ref{Th transfer subgroups} and the closed
graph theorem.

Case 3: $G$ is a group of bounded order. In this case, the remark
following Theorem \ref{Th struct group bounded} allows us to claim
that $G$ contains a subgroup isomorphic to the direct sum
$\oplus_{1}^\infty \mathbb{Z}/q\mathbb{Z}$ where $q$ is a prime
integer. We conclude with Theorem \ref{Th transfer subgroups} and
Proposition \ref{Prop incond oplus}.
\end{proof}

Recall a particular case of \cite[Theorem 1]{DJ}. We give an
independent proof of this result.
\begin{prop}
\label{Th incond compact groups} Let $G$ be an infinite compact
abelian group and $X$ be a Banach space. We have a canonical
isomorphism $M_2(G,X)= M_2(G)$ if and only if the space $X$ is
isomorphic to a Hilbert space.
\end{prop}

\begin{proof}
Let $G$ be an infinite compact group. Suppose that
$M_2(G,X)=M_2(G)$.

Case 1: The discrete group $\widehat{G}$ is not a torsion group.
Then $\widehat{G}$ contains a copy of $\mathbb{Z}$, the additive
group of the integers. Note that we have
$G/\mathbb{Z}^{\perp}=\widehat{\mathbb{Z}}=\mathbb{T}$
isomorphically. By Proposition \ref{Prop transfer quotient}, we
deduce that $M_2(\T,X)=M_2(\T)$. By Proposition \ref{Prop incond T},
we deduce that the Banach space $X$ is isomorphic to a Hilbert
space.

Case 2: The group $\widehat{G}$ is a torsion group, but contains
elements of arbitrarily large order. We may therefore assume that
there is a sequence $G_1,G_2,\ldots$ of cyclic subgroups of
$\widehat{G}$ of orders $n_1,n_2,\ldots$ with $n_j\xra[j \to
+\infty]{} +\infty$. Note that for any integer $j \geq 1$, we have
the following group isomorphisms
\begin{equation*}
G/G_j^{\perp}=\widehat{G_j}=\Z/n_j\Z.
\end{equation*}
Using Proposition \ref{Prop transfer quotient}, we conclude as the
case 2 of the proof of Proposition \ref{Th incond discrete}.

Case 3: $\widehat{G}$ is a group of bounded order. In this case, the
remark following Theorem \ref{Th struct group bounded} allows us to
claim that $\widehat{G}$ contains a subgroup isomorphic to the
direct sum $\oplus_{1}^\infty \mathbb{Z}/q\mathbb{Z}$ where $q$ is a
prime integer. Observe that we have the group isomorphisms
\begin{equation*}
G/(\oplus_{1}^\infty
\mathbb{Z}/q\mathbb{Z})^{\perp}=\widehat{\oplus_{1}^\infty
\mathbb{Z}/q\mathbb{Z}}=\Pi_{1}^\infty \mathbb{Z}/q\mathbb{Z}.
\end{equation*}
Using the fact that $\oplus_{1}^\infty \mathbb{Z}/q\mathbb{Z}$ is a
subgroup of $\Pi_{1}^\infty \mathbb{Z}/q\mathbb{Z}$, the result
follows by applying Proposition \ref{Prop transfer quotient},
Theorem \ref{Th transfer subgroups} and Proposition \ref{Prop incond
oplus}.
\end{proof}

The next result is the principal result of this section.
\begin{thm}
\label{Th incond locally} Let $G$ be an infinite locally compact
abelian group and $X$ a Banach space. We have a canonical
isomorphism $M_2(G,X)=M_2(G)$ if and only if the Banach space $X$ is
isomorphic to a Hilbert space.
\end{thm}

\begin{proof}
By Theorem \ref{Th structure}, the group $G$ is isomorphic to a
product $\R^n\times G_0$ where $G_0$ is a locally compact abelian
group containing a compact subgroup $K$ such that $G_0/K$ is
discrete. Suppose $n\geq 1$. If $M_2(G,X)=M_2(G)$, by Theorem
\ref{Th transfer subgroups}, we deduce a canonical isomorphism
$M_2(\R,X)=M_2(\R)$. Hence, by Proposition \ref{Th incond R and Z},
we conclude that the Banach space $X$ is isomorphic to a Hilbert
space. If the group $K$ is infinite, we apply a similar reasoning by
using Proposition \ref{Th incond compact groups} instead of
Proposition \ref{Th incond R and Z}. If $n=0$ and if $K$ is finite
then it is not difficult to see that $G$ is discrete. In this case
we conclude with Proposition \ref{Th incond discrete}.
\end{proof}

\section{Bounded Fourier multipliers which are not completely bounded}
In this section, we prove that if $1<p<\infty$, $p\not=2$, there
exists a bounded Fourier multiplier on $L^p(G)$ which is not
completely bounded where $G$ is an infinite locally compact abelian
group. The cases of groups $\R$, $\Z$ and infinite compact abelian
groups are already known. We start by extending these results to the
discrete group $\oplus_{1}^\infty \mathbb{Z}/q\mathbb{Z}$, where
$q\geq 2$ is an integer. In the proof of this result, we will use
the notations introduced before Proposition \ref{Prop incond oplus}.

\begin{prop}
\label{Prop mult non cb on oplus} Suppose $1<p<\infty$, $p\not=2$.
Let $q\geq 2$ an integer. There exists a bounded Fourier multiplier
on $L^p\big(\oplus_{1}^\infty \mathbb{Z}/q\mathbb{Z}\big)$ which is
not completely bounded.
\end{prop}

\begin{proof}
By Theorem \ref{Th transfer subgroups} and the closed graph theorem,
it suffices to prove that there exists contractive Fourier
multipliers on the group $\Omega_q^n\times
\Omega_q^n=(\mathbb{Z}/q\mathbb{Z}\times\cdots
\times\mathbb{Z}/q\mathbb{Z})\times
(\mathbb{Z}/q\mathbb{Z}\times\cdots \times\mathbb{Z}/q\mathbb{Z})$
with arbitrary large completely bounded norms in $n$. By Theorem
\ref{Th transfer Sidon} and Lemma \ref{Lemma epsq is a Sidon set},
we have
\begin{equation*}
\Bgnorm{\sum_{i,j=1}^n \varepsilon_i \ot \varepsilon_{j}\ot
x_{ij}}_{\Rad(\Rad(S^p))} \approx\ \Bgnorm{\sum_{i,j=1}^n
\varepsilon_{i,q} \ot \varepsilon_{j,q} \ot
x_{ij}}_{L^p(\Omega_{q}^n \times \Omega_{q}^n,S^p)},\qquad
n\in\mathbb{N}, x_{ij}\in S^p.
\end{equation*}
We let $\mathcal{R}^{p}_{2,q}$ denote the closed span of the
$\varepsilon_{i,q} \ot \varepsilon_{j,q}$'s in $L^p\big(\Omega_{q}^n
\times \Omega_{q}^n\big)$ where $1\leq i,j\leq n$. For any family
$\tau=(t_{ij})_{i,j\geq1}$ of complex numbers we consider the linear
map
\begin{equation*}
\begin{array}{cccc}
 T_\tau:   &             \mathcal{R}^{p}_{2,q}                        &  \longrightarrow   & L^p(\Omega_{q}^n \times \Omega_{q}^n)   \\
           &  \varepsilon_{i,q} \ot \varepsilon_{j,q}   &  \longmapsto       &  t_{ij}\varepsilon_{i,q}\ot \varepsilon_{j,q}.  \\
\end{array}
\end{equation*}
Note that we have
\begin{equation*}
\Bgnorm{\sum_{i,j=1}^n \alpha_{ij} \varepsilon_{i,q}\ot
\varepsilon_{j,q}}_{L^p(\Omega_{q}^n \times \Omega_{q}^n)} \approx
\Bgnorm{\sum_{i,j=1}^n \alpha_{ij} \varepsilon_i\ot
\varepsilon_{j}}_{L^p(\Omega_0 \times \Omega_0)} \approx
\Bigg(\sum_{i,j=1}^n |\alpha_{ij}|^2\Bigg)^{\frac{1}{2}},\qquad
n\in\mathbb{N}, \alpha_{ij}\in \mathbb{C}
\end{equation*}
(see \cite[Lemma 2.1]{Pis1} or \cite[page 455]{Def}). Then, for any
complex numbers $\alpha_{ij}\in \C$, we have
\begin{align*}
\Bgnorm{T_\tau\Bigg(\sum_{i,j=1}^n \alpha_{ij} \varepsilon_{i,q}\ot
\varepsilon_{j,q}\Bigg)}_{L^p(\Omega_{q}^n \times \Omega_{q}^n)}
   &=   \Bgnorm{\sum_{i,j=1}^n t_{ij}\alpha_{ij} \varepsilon_{i,q}\ot \varepsilon_{j,q}}_{L^p(\Omega_{q}^n \times \Omega_{q}^n)} \\
   &\approx \Bigg(\sum_{i,j=1}^n |t_{ij}\alpha_{ij}|^2\Bigg)^{\frac{1}{2}} \\
   &\lesssim  \sup_{1\leq i,j\leq n}|t_{ij}|\ \Bgnorm{\sum_{i,j=1}^n \alpha_{ij}\ot \varepsilon_{i,q}\ot \varepsilon_{j,q}}_{L^p(\Omega_{q}^n \times \Omega_{q}^n)}.
\end{align*}
Consequently, we have $\norm{T_\tau}_{\mathcal{R}^{p}_{2,q}
\xra{}L^p(\Omega_q^{n}\times \Omega_q^{n})}  \lesssim
\norm{\tau}_{\infty}$. Since $S^p$ does not have the property
$(\alpha)$ there exists complex numbers $t_{ij}\in \C$ with
$|t_{ij}|=1$
and large $\bnorm{T_\tau\ot Id_{S^p}}$. 
Now, using the canonical bounded projection from
$L^p\big(\Pi_{1}^{\infty} \mathbb{Z}/q\mathbb{Z}\big)$ on the
closure of ${\rm Span}\bigl\{\varepsilon_{i,q}\ |\ i\geq 1 \bigr\}$
in the space $L^p\big(\Pi_{1}^{\infty} \mathbb{Z}/q\mathbb{Z}\big)$,
we see that there exists a bounded projection from
$L^p\big(\Pi_{1}^{\infty} \mathbb{Z}/q\mathbb{Z}\times
\Pi_{1}^{\infty} \mathbb{Z}/q\mathbb{Z}\big)$ on
$\mathcal{R}^{p}_{2,q}$. Applying the inclusion map
$L^p\big(\Omega_q^n\times \Omega_q^n\big)\to
L^p\big(\Pi_{1}^{\infty} \mathbb{Z}/q\mathbb{Z}\times
\Pi_{1}^{\infty} \mathbb{Z}/q\mathbb{Z}\big)$ we obtain a bounded
projection from $L^p\big(\Omega_q^n \times \Omega_q^n\big)$ on
$\mathcal{R}^{p}_{2,q}$ with a norm which is bounded independently
of $n$. Finally, by composing with this projection, we obtain
contractive Fourier multipliers on the group $\Omega_q^n\times
\Omega_q^n$ with arbitrary completely bounded norms in $n$.
\end{proof}

Now, we can state and prove the second main result of this paper.
\begin{thm}
\label{Th noncb locally} Suppose $1<p<\infty$, $p\not=2$. Let $G$ be
an infinite locally compact abelian group. There exists a bounded
Fourier multiplier on $L^p(G)$ which is not completely bounded.
\end{thm}

\begin{proof}
The proof is similar to the ones of Proposition \ref{Th incond
discrete} and Theorem \ref{Th incond locally}. The case of a
discrete group of torsion need some minor modifications. We prove it
by a reasoning similar to the one used in the proof of Proposition
\ref{Th incond discrete} using the conditional expectation defined
by (\ref{def esperance}) as an operator $\mathbb{E}_n\colon
L^p\big([0,1]\big) \to L^p\big([0,1]\big)$ and using the isometric
map $J_n\colon \ell^p_n \xra{} \mathbb{E}_n \big(L^p([0,1])\big)$
defined by, if $u \in \ell^p_n$
\begin{equation*}
J_n(u)=n^{\frac{1}{p}}\sum_{k=0}^{n-1} u_k
1_{[\frac{k}{n},\frac{k+1}{n}[}.
\end{equation*}
\end{proof}

\begin{remark}
Using the fact the space $S^p$ does not have property $(\alpha)$ if
$1<p<\infty$, $p\not=2$ and the method used at the beginning of the
proof of Proposition \ref{Prop incond T}, anyone can give a proof of
Theorem \ref{Th noncb locally} for the case $G=\T$. The more general
case where $G$ is an infinite compact abelian group can also be
obtained with the method of the proof of Proposition \ref{Th incond
compact groups}.
\end{remark}

\begin{remark}
Recall the following classical result of S.
Kwapie\'{n} \cite{Kwa2}. Suppose $1<p<\infty$. A Banach space $X$ is
isomorphic to an $SQL^p$-space, i.e a subspace of a quotient of an
$L^p$-space if and only if for any measure space $\Omega$ and any
bounded operator $T\colon L^p(\Omega) \to L^p(\Omega)$, the operator
$T\ot Id_X\colon L^p(\Omega,X)\to L^p(\Omega,X)$ is bounded. The
results of Section 3 and of this section lead to the general
following open question. Let $X$ be a Banach space and $G$ be an
infinite locally compact abelian group. If we have a canonical
isomorphism $M_p(G,X)=M_p(G)$, do we have an isomorphism from the
Banach space $X$ on an $SQL^p$-space?
\end{remark}

\section{Unconditionality and Schur multipliers}

Suppose $1<p<\infty$. In this section, we use the notation $S^p_\Z=
S^p\big(\ell^2_\Z\big)$ and $S^p_{\Z \times \mathbb{N}}=
S^p\big(\ell^2_{\Z \times \mathbb{N}}\big)$. Recall that a Schur
multiplier on $S^p$ is a linear map $M_A\colon S^p\rightarrow S^p$
defined by a scalar matrix $A$ such that $M_A(B)=[a_{ij}b_{ij}]$
belongs to $S^p$ for any $B\in S^p$. We have a similar notion for
$S^p_\Z$. In the sequel, $(\varepsilon_{ij})_{i,j\geq 1}$ denotes a
doubly indexed family of independent Rademacher variables.

The paper \cite{Lee} contains the following result:
\begin{thm}
\label{Th Lee} Let $E$ be an operator space. Then $E$ is completely
isomorphic to an operator Hilbert space $OH(I)$ for some index set
$I$ if and only if we have an equivalence
\begin{equation*}
\Bgnorm{\sum_{i,j=1}^{n}e_{ij}\ot x_{ij}}_{S^2(E)}\approx
\Bgnorm{\sum_{i,j=1}^{n}\varepsilon_{ij} \ot
x_{ij}}_{\Rad(E)},\qquad n\in \mathbb{N}, x_{ij}\in E.
\end{equation*}
\end{thm}
First, we show a link between a property of the Banach space
$S^2(E)$ and a property of the operator space $E$.
\begin{prop}
\label{Th equivalence OH S2E}
Let $E$ be an operator space. The
following assertions are equivalent.
\begin{itemize}
  \item The Banach space $S^2(E)$ is isomorphic to a Hilbert space.
  \item The operator space $E$ is completely isomorphic to an operator Hilbert space $OH(I)$ for some index set $I$.
\end{itemize}
\end{prop}

\begin{proof}
Suppose that $S^2(E)$ is isomorphic to a Hilbert space. By Lemma
\ref{Lemma formule RAdRAd} we have
\begin{equation*}
\Bgnorm{\sum_{i,j=1}^{n}e_{ij}\ot
x_{ij}}_{S^2(E)}=\Bgnorm{\sum_{i,j=1}^{n}\varepsilon_i\ot\varepsilon_j\ot
e_{ij} \ot x_{ij} }_{\Rad(\Rad(S^2(E)))},\qquad n\in \mathbb{N},
x_{ij}\in E.
\end{equation*}
Moreover, the Banach space $\Rad(S^2(E))$ is also isomorphic to a
Hilbert space. Hence, for any integer $n\in \mathbb{N}$ and any
$x_{ij}\in E$, we deduce that
\begin{align*}
\Bgnorm{\sum_{i,j=1}^{n}e_{ij}\ot x_{ij}}_{S^2(E)}
&\approx\Bigg(\sum_{i,j=1}^{n}\norm{e_{ij}\ot x_{ij}}_{S^2(E)}^2\Bigg)^{\frac{1}{2}}\\
&=\Bigg(\sum_{i,j=1}^{n}\norm{ x_{ij}}_{E}^2 \Bigg)^{\frac{1}{2}}.
\end{align*}
The space $E$ is a closed subspace of $S^2(E)$. Hence it is
isomorphic to a Hilbert space. Then we conclude that
\begin{equation*}
\Bgnorm{\sum_{i,j=1}^{n}e_{ij}\ot
x_{ij}}_{S^2(E)}\approx\Bgnorm{\sum_{i,j=1}^{n}\varepsilon_{ij} \ot
x_{ij}}_{\Rad(E)},\qquad n\in \mathbb{N},x_{ij}\in E.
\end{equation*}
By Theorem \ref{Th Lee}, we deduce that $E$ is completely isomorphic
to $OH(I)$ for some index set $I$. The reverse implication is
obvious.
\end{proof}

We need the next theorem \cite[Remark 3.1]{NeR}.
\begin{thm}
\label{Transfer Neuwirth inf} Let $E$ be an operator space and
$\varphi\colon \Z\to \C $ be a function. Consider the infinite
matrix $A=[\varphi_{i-j}]_{i,j\in \Z}$. If the map $M_A\ot Id_{E}$
is bounded on $S^2_\Z(E) $ then the map $M_{\varphi}\ot Id_{E}$ is
bounded on $L^2(\T,E)$ and we have
\begin{equation*}
\bnorm{M_{\varphi}\ot Id_{E}}_{L^2(\T,E)\to L^2(\T,E)}\leq
\bnorm{M_A \ot Id_{E}}_{S^2_\Z(E)}.
\end{equation*}
\end{thm}

The following result shows that if the matricial units form a
`unconditional system' of $S^2(E)$ then the operator space $E$ is
completely isomorphic to an operator Hilbert space.
\begin{thm}
\label{Th incond Schur} Let $E$ be an operator space. The following
assertions are equivalent.
\begin{itemize}
  \item The Banach space $S^2(E)$ has property $(\alpha)$.
  \item There exists a positive constant $C$ such that
\begin{equation}\label{incond Schur}
\Bgnorm{\sum_{i,j=1}^n t_{ij}e_{ij}\ot x_{ij}}_{S^2(E)}\leq C
\sup_{1\leq i,j\leq n} |t_{ij}| \Bgnorm{\sum_{i,j=1}^n e_{ij}\ot
x_{ij}}_{S^2(E)}
\end{equation}
for any integer $n\in \mathbb{N}$, any complex numbers $\ t_{ij}\in
\C$ and any $\ x_{ij}\in E$.
  \item The operator space $E$ is completely isomorphic to an operator Hilbert space $OH(I)$ for some index set $I$.
\end{itemize}
\end{thm}

\begin{proof}
Suppose that the Banach space $S^2(E)$ has property $(\alpha)$. For
any integer $n\in \mathbb{N}$, any $y_{ij}\in S^2(E)$ and any
$t_{ij}\in \mathbb{C}$ we have
\begin{equation*}
\Bgnorm{\sum_{i,j=1}^nt_{ij}\varepsilon_i\ot\varepsilon_j\ot
y_{ij}}_{\Rad(\Rad(S^2(E)))} \lesssim \sup_{1\leq i,j\leq n}
|t_{ij}|\Bgnorm{\sum_{i,j=1}^n\varepsilon_i\ot\varepsilon_j\ot
y_{ij}}_{\Rad(\Rad(S^2(E)))}.
\end{equation*}
For any $1\leq i,j\leq n$, let $x_{ij}$ be an element of $E$ . Using
$y_{ij}=e_{ij}\ot x_{ij}$, we obtain
\begin{equation*}
\Bgnorm{\sum_{i,j=1}^nt_{ij}\varepsilon_i\ot\varepsilon_j\ot
e_{ij}\ot x_{ij}}_{\Rad(\Rad(S^2(E)))} \lesssim \sup_{1\leq i,j\leq
n} |t_{ij}|\Bgnorm{\sum_{i,j=1}^n\varepsilon_i\ot\varepsilon_j\ot
e_{ij}\ot x_{ij}}_{\Rad(\Rad(S^2(E)))}.
\end{equation*}
By Lemma \ref{Lemma formule RAdRAd}, we conclude that
\begin{equation*}
\Bgnorm{\sum_{i,j=1}^n t_{ij} e_{ij}\ot x_{ij}}_{S^2(E)} \lesssim
\sup_{1\leq i,j\leq n} |t_{ij}| \Bgnorm{\sum_{i,j=1}^n e_{ij}\ot
x_{ij}}_{S^2(E)}.
\end{equation*}

Now suppose that the inequality (\ref{incond Schur}) is true. Using
the completely isometric isomorphisms
\begin{equation*}
S^2_\Z\big(S^2(E)\big)=S^2_{\mathbb{Z} \times
\mathbb{N}}(E)=S^2(E),
\end{equation*}
it is easy to see that we have
\begin{equation}\label{equafinal}
\Bgnorm{\sum_{i,j=-n}^nt_{ij}e_{ij}\ot x_{ij}}_{S_\Z^2(S^2(E))}\leq
C \sup_{-n\leq i,j\leq n} |t_{ij}| \Bgnorm{\sum_{i,j=-n}^{n}
e_{ij}\ot x_{ij}}_{S^2_\Z(S^2(E))}
\end{equation}
for any integer $n$, any complex numbers $t_{ij}\in \C$ and any
$x_{ij}\in S^2(E)$. Let $\varphi:\Z \to \C$ be a function with
finite support. By (\ref{equafinal}), the map $M_A\ot Id_{S^2(E)}$
on $S_{\Z}^2\big(S^2(E)\big)$ associated with the matrix
$A=[\varphi_{i-j}]_{i,j\in \Z}$ is bounded with
\begin{equation*}
\bnorm{M_A \ot Id_{S^2(E)}}_{S_{\Z}^2(S^2(E)) \to
S_{\Z}^2(S^2(E))}\leq C.
\end{equation*}
Then by Theorem \ref{Transfer Neuwirth inf}, we deduce that the map
$M_\varphi \ot Id_{S^2(E)}$ is bounded on $L^2\big(\T,S^2(E)\big)$
and that we have
\begin{equation*}
\bnorm{M_\varphi\ot Id_{S^2(E)}}_{L^2(\T,S^2(E))\to
L^2(\T,S^2(E))}\leq \bnorm{M_A\ot Id_{S^2_\Z(E)}}_{S^2_\Z(S^2(E))
\to S^2_\Z(S^2(E))}\leq C.
\end{equation*}
For any sequence $(x_k)$ of elements of $E$, we deduce that
\begin{equation*}
\Bgnorm{\sum_{k=-\infty}^{+\infty} \varphi(k)e^{2\pi ik \cdot}\ot
x_k}_{L^2(S^2(E))}\leq C \sup_{k\in \Z} \left|\varphi(k)\right|
\Bgnorm{\sum_{k=-\infty}^{+\infty} e^{2\pi ik \cdot}\ot
x_k}_{L^2(S^2(E))}.
\end{equation*}
By Theorem \ref{Th incond T}, the Banach space $S^2(E)$ is
isomorphic to a Hilbert space. Finally, by Theorem \ref{Th
equivalence OH S2E}, we infer that the operator space $E$ is
completely isomorphic to an operator Hilbert space $OH(I)$ for some
index set $I$.

The remaining implication is trivial.
\end{proof}

\begin{remark}
The results of Section 4 raise the question to prove an analog
result for Schur multipliers. Indeed, in this context, G. Pisier
conjectured that there exists a Schur multiplier which is bounded on
$S^p$ but not completely bounded if $1<p<\infty$, $p\not=2$ (see
\cite[Conjecture 8.1.12]{Pis2}).
\end{remark}


\textbf{Acknowledgment}. The author is greatly indebted to Christian
Le Merdy for many useful discussions and a careful reading. The
author would like to thank Wolfgang Arendt to encourage him to write
some results of this paper and Stefan Neuwirth for some discussions.
The author is greatly indebted to \'Eric Ricard for very fruitful
observations and to the anonymous referee for many very helpful
comments.


\footnotesize{ \n Laboratoire de Math\'ematiques, Universit\'e de
Franche-Comt\'e,
25030 Besan\c{c}on Cedex,  France\\
cedric.arhancet@univ-fcomte.fr\hskip.3cm

\small
\begin{thebibliography}{79}

\bibitem[Aba]{AbA}
Y. Abramovich and C. Aliprantis.
\newblock {\em An invitation to operator theory}.
\newblock American Mathematical Society, Providence, 2002.

\bibitem[ArB]{AB}
W. Arendt and S. Bu.
\newblock Fourier series in Banach spaces and maximal regularity.
\newblock Vector measures, integration and related topics, 21--39, Oper. Theory Adv. Appl., 201, Birkhäuser Verlag, Basel, 2010.

\bibitem[Arh]{Arh1}
C. Arhancet.
\newblock On Matsaev's conjecture for contractions on noncommutative $L^p$-spaces.
\newblock To appear in Journal of Operator theory. Preprint, arXiv:1009.1292,
2010.

\bibitem[BeG]{BG}
E. Berkson and T. Gillespie.
\newblock Spectral decompositions and harmonic analysis on UMD Banach spaces.
\newblock Studia Math. 112, 13--49, 1994.

\bibitem[ClP]{ClP}
P. Clement, B. de Pagter, F. A. Sukochev and H. Witvliet.
\newblock Schauder decomposition and multiplier theorems.
\newblock Studia Math. 138, no. 2, 135--163, 2000.

\bibitem[CoW]{CoW}
R. Coifman and G. Weiss.
\newblock {\em Transference methods in analysis.}
\newblock Conference Board of the Mathematical Sciences Regional Conference Series in Mathematics, No. 31. American Mathematical Society, Providence, R.I., 1976.

\bibitem[Def]{Def}
A. Defant and K. Floret.
\newblock Tensor Norms and Operators Ideals.
\newblock North-Holland Mathematics Studies, 176, Amsterdam, 1993.

\bibitem[DeJ]{DJ}
M. Defant and M. Junge.
\newblock Unconditional Orthonormal Systems.
\newblock Math. Nachr. 158, 233--240, 1992.

\bibitem[DeL]{DeL}
K. De Leeuw.
\newblock On $L_p$ multipliers.
\newblock Ann. of Math. 81, 364--379, 1965.

\bibitem[Der]{Der}
A. Derighetti.
\newblock Convolution operators on groups.
\newblock  Lecture Notes of the Unione Matematica Italiana, 11. Springer, Heidelberg; UMI, Bologna, 2011.

\bibitem[DPR]{DPR}
B. De Pagter and W. J. Ricker.
\newblock $C(K)$-representations and $R$-boundedness.
\newblock J. Lond. Math. Soc. (2) 76, no. 2, 498--512, 2007.

\bibitem[DJT]{DJT}
J. Diestel, H. Jarchow and A. Tonge.
\newblock Absolutely summing operators.
\newblock Cambridge Studies in Advanced Mathematics, 43. Cambridge University Press, 1995.

\bibitem[Fol]{FO}
G. B. Folland.
\newblock A course in abstract harmonic analysis.
\newblock Studies in Advanced Mathematics. CRC Press, Boca Raton, FL, 1995.

\bibitem[Har]{Har}
A. Harcharras.
\newblock  Fourier analysis, Schur multipliers on $S^p$ and noncommutative $\Lambda(p)$-sets.
\newblock  Studia math. 137, 203--260, 1999.

\bibitem[HeR]{HR}
E. Hewitt and K. A. Ross.
\newblock Abstract harmonic analysis. Vol. I.
Structure of topological groups, integration theory, group
representations. Second edition.
\newblock  Grundlehren der Mathematischen Wissenschaften, 115. Springer-Verlag, Berlin-New York, 1979.

\bibitem[Jod]{Jod}
M. Jodeit.
\newblock Restrictions and extensions of Fourier multipliers.
\newblock Studia Math. 34, 215--226, 1970.

\bibitem[KLM]{KLM}
C. Kriegler and C. Le Merdy.
\newblock Tensor extension properties of $C(K)$-representations and applications to unconditionality.
\newblock J. Aust. Math. Soc. 88, no. 2, 205--230, 2010.

\bibitem[KuW]{KW}
P. C. Kunstmann  and L.  Weis.
\newblock Maximal $L_p$-regularity for parabolic equations, Fourier multiplier theorems and $H^\infty$-functional calculus.
\newblock pp. 65-311 in {\it Functional analytic methods for evolution equations}, Lect. Notes in Math. 1855, Springer, 2004.

\bibitem[Kwa1]{Kwa}
S. Kwapie\'{n}.
\newblock Isomorphic characterizations of inner product spaces by orthogonal
series with vector valued coefficients.
\newblock Studia Math. 44, 583--595, 1972.

\bibitem[Kwa2]{Kwa2}
S. Kwapie\'{n}.
\newblock On operators factorizable through $L_{p}$-space.
\newblock Bull. Soc. Math. France, Mémoire 31--32, 215--225, 1972.

\bibitem[Lar]{Lar}
R. Larsen.
\newblock An introduction to the theory of multipliers.
\newblock Springer-Verlag, 1971.

\bibitem[Lee]{Lee}
H. H. Lee.
\newblock Type and cotype of operator spaces.
\newblock Studia Math. 185, no. 3, 219--247, 2008.


\bibitem[LoR]{LoR}
J. M. L\'opez and K. A. Ross.
\newblock {\em Sidon sets}.
\newblock Lecture Notes in Pure and Applied Mathematics, Vol. 13. Marcel Dekker, Inc., New York,
1975.

\bibitem[NeR]{NeR}
S. Neuwirth and \'{E}. Ricard.
\newblock Transfer of Fourier multipliers into Schur multipliers and sumsets in a discrete group.
\newblock Canad. J. Math. 63, no. 5, 1161--1187, 2011.

\bibitem[PiW]{PiW}
A. Pietsch and J. Wenzel.
\newblock Orthonormal systems and Banach space geometry.
\newblock Encyclopedia of Mathematics and its Applications 70, Cambridge, 1998.

\bibitem[Pis1]{Pis1}
G. Pisier.
\newblock Some results on Banach spaces without local unconditional structure.
\newblock Compositio Math. 37, no. 1, 3--19, 1978.

\bibitem[Pis2]{Pis2}
G. Pisier.
\newblock {\it Non-commutative vector valued $L_p$-spaces and completely $p$-summing maps}.
\newblock Astérisque, 247, 1998.

\bibitem[Pis3]{Pis3}
G. Pisier.
\newblock {\em Introduction to operator space theory}.
\newblock Cambridge University Press, Cambridge, 2003.

\bibitem[Pis4]{Pis4}
G. Pisier.
\newblock {\it The operator {H}ilbert space {${\rm OH}$}, complex interpolation and tensor norms}.
\newblock Mem. Amer. Math. Soc. 122, 1996.

\bibitem[Pis5]{Pis5}
G. Pisier.
\newblock Les inégalités de Khintchine-Kahane, d'après C. Borell. (French).
\newblock Séminaire sur la Géométrie des Espaces de Banach (1977--1978), Exp. No. 7, École Polytech., Palaiseau, 1978.


\bibitem[Sae]{Sae}
S. Saeki.
\newblock Translation invariant operators on groups.
\newblock Tôhoku Math. J. (2) 22 1970 409--419, 1970.

\bibitem[SWS]{SWS}
F. Schipp, W. R. Wade and P. Simon.
\newblock Walsh series. An introduction to dyadic harmonic analysis.
\newblock Adam Hilger, Bristol, 1990.






\end{thebibliography}
\end{document}